\documentclass[11pt, a4paper]{amsart}
\usepackage{amssymb,amsmath}
\usepackage{hyperref}
\usepackage[capitalize]{cleveref}
\usepackage{multicol} 

\numberwithin{equation}{section}

\newtheorem{thm}[equation]{Theorem}
\newtheorem{prop}[equation]{Proposition}
\newtheorem{lem}[equation]{Lemma}
\newtheorem{cor}[equation]{Corollary}

\newtheorem{FGL}[equation]{Lemma}

\crefformat{thm}{Theorem~#2#1#3}
\crefname{thm}{theorem}{theorems}
\Crefname{thm}{Theorem}{Theorems}
\crefformat{lem}{Lemma~#2#1#3}
\crefname{lem}{lemma}{lemmas}
\Crefname{lem}{Lemma}{Lemmas}
\crefformat{cor}{Corollary~#2#1#3}
\crefname{cor}{corollary}{corollaries}
\Crefname{cor}{Corollary}{Corollaries}
\crefformat{prop}{Proposition~#2#1#3}
\crefname{prop}{proposition}{propositions}
\Crefname{prop}{Proposition}{Propositions}
\crefmultiformat{prop}{Propositions~#2#1#3}{ and~#2#1#3}{, #2#1#3}{ and~#2#1#3}

\crefformat{FGL}{the Factor Group Lemma~(#2#1#3)}

\theoremstyle{definition}
\newtheorem*{ack}{Acknowledgment}
\newtheorem{assump}[equation]{Assumption}
\newtheorem{defn}[equation]{Definition}
\newtheorem{notation}[equation]{Notation}
\newtheorem{rem}[equation]{Remark}

\newcounter{case}
\newenvironment{case}[1][\unskip]{\refstepcounter{case}\bf
\medbreak \noindent Case \thecase\ #1. \it}{\unskip\upshape}
\crefformat{case}{Case~#2#1#3}
\newcounter{caseholder}
\numberwithin{case}{caseholder}
\renewcommand{\thecase}{\arabic{case}}
\newcounter{subcase}
\numberwithin{subcase}{case}
\newenvironment{subcase}[1][\unskip]{\refstepcounter{subcase}\bf
\smallskip \indent Subcase \thesubcase\ #1. \it}{\unskip\upshape}
\crefformat{subcase}{Subcase~#2#1#3}
\crefname{subcase}{subcase}{subcases}
\Crefname{subcase}{Subcase}{Subcases}
\newcounter{subsubcase}
\numberwithin{subsubcase}{subcase}
\newenvironment{subsubcase}[1][\unskip]{\refstepcounter{subsubcase}\bf
\smallskip \indent \hskip\parindent Subsubcase \thesubsubcase\ #1. \it}{\unskip\upshape}
\crefformat{subsubcase}{Subsubcase~#2#1#3}
\Crefname{subsubcase}{Subsubcase}{Subsubcases}
\crefname{subsubcase}{subsubcase}{subsubcases}
\newcounter{subsubsubcase}
\numberwithin{subsubsubcase}{subsubcase}

\crefformat{subsubsubcase}{Subsubsubcase~#2#1#3}

\newcommand{\pref}[1]{(\ref{#1})}

\newcommand{\fullcref}[2]{\cref{#1}\pref{#1-#2}}

\renewcommand{\pmod}[1]{\ (\mathrm{mod}~{#1})} 

\newcommand{\csee}[1]{(see \cref{#1})}

\DeclareMathOperator{\Aut}{Aut}
\DeclareMathOperator{\Cay}{Cay}
\DeclareMathOperator{\lcm}{lcm}
\DeclareMathOperator{\Norm}{Norm}

\newcommand{\CC}{\mathbb{C}}
\newcommand{\cyclic}[1]{\mathcal{C}_{#1}}
\newcommand{\voltage}{\mathbb{V}}
\newcommand{\NN}{\mathbb{N}}
\newcommand{\ZZ}{\mathbb{Z}}

\DeclareFontFamily{U}{mathx}{}
\DeclareFontShape{U}{mathx}{m}{n}{<-> mathx10}{}
\DeclareSymbolFont{mathx}{U}{mathx}{m}{n}
\DeclareMathAccent{\widecheck}{0}{mathx}{"71}

\renewcommand{\MR}[1]{\href{http://www.ams.org/mathscinet-getitem?mr=#1}{MR\,#1}}
\newcommand{\doi}[1]{\href{https://dx.doi.org/#1}{doi:#1}}

\makeatletter
\newcommand{\noprelistbreak}{\smallskip\@nobreaktrue\nopagebreak} 
\makeatother

\makeatletter
\def\@setaddresses{\par\bigbreak
\begingroup\normalfont
  \def\author##1{}%
  \def\\{}%
  \def\address##1##2{}%
  \def\curraddr##1##2{}%
  \def\email##1##2{\begingroup
    \@ifnotempty{##2}{\par\smallskip
      \@ifnotempty{##1}{\ignorespaces##1\unskip}\/:\space 
      \ttfamily##2\par}\endgroup}%
  \def\urladdr##1##2{}%
    \noindent \emph{Email addresses:} 
  \addresses
  \endgroup
}
\makeatother

\begin{document}

	
\title{Cayley Graphs of Order $8pq$ are Hamiltonian}

\author[F.\,Abedi, D.\.W.\,Morris, J.\,Rezaee, and M.\,R.\,Salarian]{Fateme Abedi\rlap,$^1$ Dave Witte Morris\rlap,$^2$ 
	Javanshir Rezaee\rlap,$^1$ and M.~Reza Salarian$^{1,3}$}

\footnotetext[1]{Faculty of Mathematical Sciences and Computer, Kharazmi University, 50 Taleghani Ave., Tehran, 1561836314, I.\,R.\,Iran}
\footnotetext[2]{Department of Mathematics and Computer Science, University of Lethbridge, 
4401 University Drive, Lethbridge, Alberta, T1K\,3M4, Canada}
\footnotetext[3]{A.\,Renyi Institute of Mathematics, Hungarian Academy of Sciences, P.\,O.\,Box 127, H-1364 Budapest, Hungary}

\email[F.\,Abedi]{std\_f.abedi@khu.ac.ir, f.abedi1126@gmail.com}
\email[D.\,W.\,Morris]{dmorris@deductivepress.ca, https://deductivepress.ca/dmorris}
\email[J.\,Rezaee]{jre927@gmail.com}
\email[M.\,R.\,Salarian]{salarian@khu.ac.ir, salarian78@gmail.com}



\begin{abstract}
We give a computer-assisted proof that if $G$ is a finite group of order $8pq$, where $p$ and $q$ are distinct primes, then every connected Cayley graph on $G$ has a hamiltonian cycle.
\end{abstract}

{\mathversion{bold}
\maketitle
}


\section{Introduction}

Numerous papers show that all connected Cayley graphs of certain orders are hamiltonian. (See \cref{CayleyDefn} for a definition of the term ``Cayley graph\rlap.'') Several of these results are collected in the following theorem, which is an updated version of \cite[Thm.~1.2]{KutnarEtAl}. 

\begin{thm}[cf.\ {\cite[Thm.~1.2]{Maghsoudi-6pq}}] \label{CertainOrders}
If $G$ is a finite group with $|G| > 2$, and\/ $|G| $ has any of the following forms \textup(where $p$, $q$, $r$ and $s$ are distinct primes, and $k$ is a positive integer\textup), then every connected Cayley graph on~$G$ has a hamiltonian cycle:
\begin{multicols}{2}
	\begin{enumerate}
	\item \label{CertainOrders-kp}
	$kp$, where $k\le  47$,
	\item \label{CertainOrders-kpq}
	$kpq$, where $k\le 7$ or $k=9$,
	\item $pqr$,
	\item \label{CertainOrders-pqrs}
	$pqrs$ if $p$, $q$, $r$ and $s$ are odd,
	\item $kp^2$, where $k\le 4$,
	\item $kp^3$, where $k\le 2$,
	\item \label{CertainOrders-pgroup}
	$p^k$.
	\end{enumerate}
\end{multicols}
\end{thm}

\begin{rem} 
	The introduction of~\cite{Maghsoudi-6pq} provides a list of the papers that were combined to make \cref{CertainOrders}, except that it does not have references for the two parts of the \lcnamecref{CertainOrders} that do not appear in \cite{Maghsoudi-6pq}'s statement of the result: see \cite[Cor.~1.5]{Morris-oddcomm} for the case $k = 9$ of part~\pref{CertainOrders-kpq}, and see \cite{Morris-pqrs} for part~\pref{CertainOrders-pqrs}. A more detailed (but outdated, and therefore incomplete) explanation of the contribution from each paper is in \cite[\S2A]{KutnarEtAl}.
\end{rem}

The purpose of this paper is to add the case $k = 8$ to part~\pref{CertainOrders-kpq} of the \lcnamecref{CertainOrders}:

\begin{thm} \label{8pqThm}
If $p$ and $q$ are distinct primes, then every connected Cayley graph of order $8pq$ has a hamiltonian cycle.
\end{thm}

\begin{rem}
This means that part~\pref{CertainOrders-kpq} of \cref{CertainOrders} can be replaced with:
	\begin{itemize}
	\item[(\ref{CertainOrders-kpq}$'$)] $kpq$, where $k\le 9$.
	\end{itemize}
\end{rem}

The proof of \cref{8pqThm} relies on an exhaustive case-by-case analysis, like most other parts of \cref{CertainOrders}. However, although almost all parts of that \lcnamecref{CertainOrders} were proved by hand (so some of the papers are long and complicated --- see, for example, the proof of the case $k = 6$ of part~\pref{CertainOrders-kpq} of the \lcnamecref{CertainOrders} in~\cite{Maghsoudi-6pq}), we will use a computer-assisted approach that is adapted from the method that was used to complete part~\pref{CertainOrders-kp} of the \lcnamecref{CertainOrders} in~\cite{MorrisWilk}. \Cref{FGL} is the main tool. See \cref{ComputerFGL} for an explanation of the technique.

\goodbreak 

Here is an outline of the paper.
\noprelistbreak 
	\begin{itemize}
	\item \Cref{PrelimSect} consists of preliminaries on several topics: hamiltonian cycles in Cayley graphs, the Factor Group Lemma, generalized dihedral groups, elementary number theory, and group theory. 
	\item \Cref{ComputerSect} describes how we use a computer to find hamiltonian cycles.
	\item \Cref{AssumpSect} spells out assumptions and notation that will be in effect for all later sections of the paper.
	\item \Cref{ElemAbelSect,SpecialDihedralSect} deal with two cases that cannot be handled by our computer programs. (However, \cref{ElemAbelSect} does use a computer program in the final \lcnamecref{ElemAbel-S=4-neq-notxp} of the proof of \cref{ElemAbel}.)
	\item \Cref{MainProofSect} proves \cref{8pqThm}.
	\end{itemize}

\begin{ack}
The project of Reza Salarian, leading to this publication, has received funding from the European Research Council (ERC) under the European Union's Horizon 2020 research and innovation program  (grant agreement No 741420).
\end{ack}

\section{Preliminaries} \label{PrelimSect}

\begin{notation}
$G$ is always a finite group.
\end{notation}

\subsection{Some basic results on hamiltonian cycles in Cayley graphs}

\begin{defn}[cf.\ {\cite[p.~34]{GodsilRoyle}}] \label{CayleyDefn}
Let $S$ be a subset of $G$.
The \emph{Cayley graph} $\Cay(G;S)$ is the graph whose vertices are the elements of~$G$, such that there is an edge joining two vertices $g$ and~$h$ if and only if $h = gs$ for some $s \in S \cup S^{-1}$ (where $S^{-1} = \{\, s^{-1} \mid s \in S \,\}$.
\end{defn}

\begin{rem}[{\cite[Rem.~2.2]{MorrisWilk}}]
Unlike most authors, we do not require $S$ to be symmetric (i.e., closed under inverses). Instead, in our notation, $\Cay(G; S)=\Cay(G;S\cup S^{-1})$.
\end{rem}

\begin{thm}[{\cite[Problem~12.17, pp.~89 and 505--506]{Lovasz-CombProbs}, \cite{KeatingWitte}, \cite{Morris-2p}}] \label{4}
Assume $|G| > 2$. Every connected Cayley graph on $G$ has a hamiltonian cycle if any of the following are true:
\noprelistbreak
	\begin{enumerate}
	\item \label{4-abelian}
	$G$ is abelian \textup(in other words, the commutator subgroup of~$G$ is trivial\/\textup),
	or
	\item \label{4-KeatingWitte}
	the commutator subgroup of $G$ is a cyclic $p$-group, for some prime~$p$,
	or
	\item \label{4-2p}
	the commutator subgroup of $G$ has order $2p$, where $p$ is an odd prime.
	\end{enumerate}
\end{thm}

The following elementary observation is well known (and was used in the proofs of almost all parts of \cref{CertainOrders}). (A generating set~$S$ of a group~$G$ is \emph{irredundant} if no proper subset of~$S$ generates~$G$, so the result follows easily from the fact that every hamiltonian cycle of a spanning subgraph is a hamiltonian cycle of the ambient graph.) When proving that all connected Cayley graphs on a group~$G$ are hamiltonian, it allows us to consider only the Cayley graphs of generating sets that are irredundant.

\begin{lem}\label{OnlyMinimal}
If there is a hamiltonian cycle in the Cayley graph of every \underline{irredundant} generating set of~$G$, then every connected Cayley graph on~$G$ has a hamiltonian cycle.
\end{lem}

The following observation is also well known.

\begin{lem}[{\cite[Lem.~2.27]{KutnarEtAl}}] \label{5}
Let $S$ generate~$G$ and let $s\in S$, such that $\left\langle s\right\rangle \mathrel{\unlhd} G$. If
	\begin{itemize}
	\item $\Cay (G/ \langle s\rangle ; S)$ has a hamiltonian cycle, and
	\item either 
		\begin{enumerate}
		\item \label{5-Z}
		$s\in Z(G)$, or
		\item \label{5-noZ}
		$Z(G)\cap \langle s\rangle=1$, or
		\item \label{5-prime}
		$|s|$ is prime,
		\end{enumerate}
	\end{itemize}
then $\Cay(G;S)$ has a hamiltonian cycle.
\end{lem}

A slight modification of the proof of part~\pref{5-noZ} of the \lcnamecref{5} establishes the following generalization:

\begin{lem} \label{Z(G)factor}
Let $S$ generate~$G$ and let $s\in S$, such that $\left\langle s\right\rangle \mathrel{\unlhd} G$. If 
	\begin{itemize}
	\item $\Cay (G/\langle s\rangle ; S)$ has a hamiltonian cycle, 
	\item $Z(G)\cap \langle s\rangle$ is a direct factor of~$\langle s \rangle$, 
	and 
	\item $|Z(G)\cap \langle s\rangle|$ is a divisor of $|G : \langle s \rangle|$,
	\end{itemize}
then $\Cay(G;S)$ has a hamiltonian cycle.
\end{lem}

\subsection{Factor Group Lemma}
By \cref{CertainOrders}(\ref{CertainOrders-kp},\ref{CertainOrders-kpq}), we know that if $N$ is any nontrivial, proper, normal subgroup of~$G$, then every connected Cayley graph on $G/N$ has a hamiltonian cycle. 
It is therefore very useful to be able to lift hamiltonian cycles from a quotient graph to the original Cayley graph. \Cref{FGL} is a fundamental result that often makes this possible.

\begin{notation}[cf., e.g., {\cite[\S2.1]{Maghsoudi-6pq}}]
For $s_1, s_2,..., s_n \in S\cup S^{-1}$:
	\begin{itemize}
	\item $(s_1, s_2, s_3,...,s_n)$ denotes the walk in $\Cay(G;S)$ that visits (in order) the vertices 
		\[ 1, \ s_1, \ s_1s_2, \ s_1s_2s_3, \ ...,\ s_1s_2...s_n . \]
	\item $(s_1, s_2, s_3,...,s_n)^k$ denotes the walk that is obtained from the concatenation of
$ k $ copies of $(s_1, s_2, s_3,...,s_n)$,
	and
	\item $(s_1,s_2,s_3,...,s_n)\# $ denotes the walk $(s_1,s_2,s_3...,s_{n-1})$ that is obtained by deleting the last term of the sequence.
	\end{itemize}
\end{notation}

%

\begin{defn}[cf.~{\cite[\S2.1.3, p.~61]{GrossTucker}}]
Suppose $N$ is a normal subgroup of $G$ and $C=(s_1,s_2,...,s_n)$ is a walk in $\Cay(G;S)$. If the walk $(s_1N,s_2N,...,s_nN)$ in $\Cay(G/N;SN/N)$ is closed, then its \emph{voltage} is the product $\mathbb{V}(C)=s_1s_2...s_n$. This is an element of $N$.
\end{defn}

\begin{FGL}[Factor Group Lemma {\cite[\S2.2]{WitteGallian}}] \label{FGL} 
Suppose that
	\begin{itemize}
	\item $S$ is a generating set of $G$,
	\item $N$ is a cyclic normal subgroup of $G$,
	\item $(s_1N,...,s_nN)$ is a hamiltonian cycle in $\Cay(G/N; S)$, 
	and
	\item the product $s_1s_2...s_n$ generates $N$.
	\end{itemize}
Then $(s_1,...,s_n)^{|N|}$ is a hamiltonian cycle in $\Cay(G;S)$.
\end{FGL}


\begin{cor}[{\cite[Cor.~2.11]{KutnarEtAl}}] \label{n1}
Suppose that
	\begin{itemize}
	\item $N$ is a cyclic, normal subgroup of $G$, such that $|N|$ is prime,
	\item $S$ is an irredundant generating set of~$G$, 
	\item there is a hamiltonian cycle in $\Cay(G/N;S)$, 
	and
	\item $s \equiv t \pmod{N}$ for some $s,t\in S\cup S^{-1}$ with $s\neq t$.
	\end{itemize}
Then there is a hamiltonian cycle in $\Cay(G;S)$.	
\end{cor}


\begin{notation}
\leavevmode
\noprelistbreak
	\begin{enumerate}
	\item $[g,h] = g^{-1} h^{-1} g h$ is the \emph{commutator} of two elements $g$ and~$h$ of~$G$.
	\item $G' = \langle \, [g,h] \mid g,h \in G \, \rangle$ is the \emph{commutator subgroup} of~$G$.
	\item We use $\cyclic{n}$ to denote a (multiplicative) cyclic group of order~$n$.
	\end{enumerate}
\end{notation}


\begin{lem}[{\cite[Cor.~2.14]{Maghsoudi-6pq}}] \label{genG'}
If $G = \langle s, t \rangle$, $G'$ is cyclic, and $\gcd \bigl( k, |G| \bigr) = 1$, then $G' = \langle [s^k, t] \rangle$.
\end{lem}

\begin{cor} \label{abelianization=Z2xZ2}
Assume $G'$ is cyclic and $G / G' \cong \cyclic2  \times \cyclic2 $. If $S$ is any $2$-element generating set of~$G$, then $\Cay(G; S)$ has a hamiltonian cycle.
\end{cor}

\begin{proof}
Write $S = \{s, t\}$. Then $(s^{-1}, t^{-1}, s, t)$ is a hamiltonian cycle in $\Cay(G/G'; S)$ whose voltage is $s^{-1} t^{-1} s t = [s,t]$. This generates~$G'$ (by \cref{genG'} with $k = 1$), so \cref{FGL} applies. 
\end{proof}

\subsection{Generalized dihedral groups}

\begin{notation}[cf.\ {\cite[Defn.~1.2]{AlspachChenDean}}]
We use $D_{2n}$ to denote the dihedral group of order $2n$. That is,
	\begin{align*}
	D_{2n} &= \left\langle f, x \mid f^2 = x^n = 1, fxf = x^{-1}\right\rangle
	. \end{align*}
(In~\cite{AlspachChenDean}, this group is called $D_n$, instead of~$D_{2n}$, but that is not consistent with the notation used in \textsf{GAP}~\cite{GAP}.)
\end{notation}

\begin{defn}[{\cite[Defn.~1.3]{AlspachChenDean}}]
A group $G$ is a \emph{generalized dihedral group} if it has
		\begin{itemize}
		\item an abelian subgroup $A$ of index 2, 
		and
		\item an element $f$ of order 2 (with $f\notin A$),
		\end{itemize}
	such that $f$ inverts every element of $A$ (i.e., $faf = a^{-1}$ for all $a\in A$).
\end{defn}

Thus, dihedral groups are the generalized dihedral groups in which $A$ is cyclic.

\begin{thm}[Alspach, Chen, and Dean {\cite[Thm.~1.8]{AlspachChenDean}}] \label{DihTypeDivBy4} 
If $G$ is a generalized dihedral group, and\/ $|G|$ is divisible by~4, then every connected Cayley graph on $G$ has a hamiltonian cycle. 
\end{thm}

\begin{rem} \leavevmode
	\begin{enumerate}
	\item We only need the special case of \cref{DihTypeDivBy4} in which the valency of the graph is~$3$, which is much easier (cf.~\cite{AlspachZhang}).
	\item Alspach, Chen, and Dean \cite{AlspachChenDean} actually proved not only that there is a hamiltonian cycle, but that the Cayley graph is hamiltonian connected (or hamiltonian laceable if it is bipartite), if the valency is at least~$3$.
	\end{enumerate}
\end{rem}

\subsection{Elementary number theory.}

We will use the following two very easy and elementary observations:

\begin{lem} \label{0modpandq}
If $p$ and~$q$ are prime numbers, and there exist $i,j \in \{0,1,2\}$, such that 
	\[ \text{$2^i p -1 \equiv 0 \pmod{q}$ 
	\ and \ 
	$2^j q -1 \equiv 0 \pmod{p}$,} \]
then $\min(p,q) \le 5$.
\end{lem}

\begin{proof}
Write $2^i p - 1 = kq$ and $2^j q - 1 = \ell p$. Then
	\begin{align} \label{0modpandq-0}
	0 
	\equiv \ell k q
	= \ell(2^i p - 1)
	= 2^i \ell p - \ell
	= 2^i (2^j q - 1) - \ell
	\equiv - (2^i + \ell)
	\qquad \pmod{q} 
	. \end{align}
Assuming, without loss of generality, that $q < p$, we also have 
	\[ \ell p = 2^j q - 1 < 2^j p \le 4p , \]
so $\ell \le 4$. Since $i \le 2$, we also have $2^i \le 4$, so $2^i + \ell \le 8$. So \pref{0modpandq-0}~implies $q \le 7$.

This obviously implies 
	\[ p \le 2^2 q - 1 \le 2^2 \cdot 7 - 1 = 27 . \]
Therefore, it is easy to see by exhaustive search that
	\[ (p, q) \in \{ 
	( 3, 2 ), 
	( 7, 2 ), 
	( 5, 3 ), 
	( 11, 3 ), 
	( 19, 5 ) 
	\}. \]
By inspection, we conclude that $\min(p, q) \le 5$, as desired.
\end{proof}

\begin{lem} \label{add3}
Assume $p,q > 3$ are two distinct prime numbers. If $x, a_1,a_2,a_3 \in \ZZ$, such that $\gcd(a_i, pq) = 1$ for all~$i$, then there is a subset~$I$ of $\{1,2,3\}$ \textup(possibly empty\textup), such that $x + \sum_{i \in I} a_i$ is relatively prime to~$pq$.
\end{lem}

\begin{proof}
Note that:
	\begin{itemize}
	\item If $\gcd(x, pq) = 1$, then we may let $I = \emptyset$.
	\item If $x \equiv 0 \pmod{pq}$, then we may let $I = \{1\}$.
	\end{itemize}
Therefore, we may assume (after interchanging $p$ and~$q$, if necessary) that $\gcd(x, pq) = p$.

Now, for all $i$, we have 
	\[ x + a_i \equiv 0 + a_i = a_i \not\equiv 0 \pmod{p} . \]
Therefore, if there is some~$i$, such that $x + a_i \not\equiv 0 \pmod{q}$, then we may let $I = \{i\}$. So we may assume, for all~$i$, that  
	\[ a_i \equiv -x \pmod{q} .\]

Since $a_3 \not\equiv 0 \pmod{p}$, we have
	\[ x + a_1 + a_2 \not\equiv x + a_1 + a_2 + a_3 \pmod{p} ,\]
so, by letting $I$ be either $\{1,2\}$ or $\{1,2,3\}$, we may arrange that 
	\[ x + \sum_{i \in I} a_i \not\equiv 0 \pmod{p} . \]
Note that we also have 
	\begin{align*}
	&x + a_1 + a_2 \equiv x + (-x) + (-x) = -x \not\equiv 0 \pmod{q} \\
	\intertext{and}
	&x + a_1 + a_2 + a_3 \equiv x + (-x) + (-x) + (-x) = -2x \not\equiv 0 \pmod{q}
	. \end{align*}
Therefore $x + \sum_{i \in I} a_i$ is relatively prime to $pq$, as desired.
\end{proof}

\Cref{n1} requires $|N|$ to be prime, but \cref{add3} yields the following analogous result that allows the cyclic normal subgroup~$N$ to have order~$pq$, which is usually the case in the proof of the main theorem. 

\goodbreak 

\begin{cor}\label{occur3}
Suppose that
	\begin{itemize}
	\item $p,q > 3$ are two distinct prime numbers,
	\item $N$ is a cyclic, normal subgroup of $G$, such that $|N| = pq$, 
	\item $S$ is a generating set of~$G$, 
	\item  $C = (s_1,s_2,\ldots, s_n)$ is a hamiltonian cycle in $\Cay(G/N;S)$, 
	\item $\langle s^{-1} t \rangle = N$ for some $s,t\in S\cup S^{-1}$,
	and
	\item $|\bigl\{\, i \mid s_i \in \{s^{\pm1}\} \,\bigr\}| \ge 3$.
	\end{itemize}
Then there is a hamiltonian cycle in $\Cay(G;S)$.	
\end{cor}

\begin{proof}
We have $t = s a$, for some~$a \in N$, such that $\langle a \rangle = N$.
Let  $J = \bigl\{\, i \mid s_i \in \{s^{\pm1}\} \,\bigr\}$. For each subset $I$ of~$J$, let $C_I$ be the hamiltonian cycle in $\Cay(G/N;S)$ that is obtained by replacing $s_i$ with~$t$ (or by~$t^{-1}$, if $s_i = s^{-1}$), for each $i \in I$. For each $i \in J$, let 
	\[ a_i = \voltage(C_{\{i\}}) \, \voltage(C^{-1})
		= \begin{cases}
			s_1 s_2 \cdots s_i a (s_1 s_2 \cdots s_i)^{-1} & \text{if $s_i = s$}, \\
			s_1 s_2 \cdots s_{i-1} a^{-1} (s_1 s_2 \cdots s_{i-1})^{-1} & \text{if $s_i = s^{-1}$}
			. \end{cases} \]
In either case, $a_i$ is a conjugate of $a$ or~$a^{-1}$, and therefore generates~$N$.

We claim that if $j \in J$ and $I \subseteq J$, such that $j \notin I$, then
	\[ \voltage(C_{I \cup \{j\}}) = a_j \cdot \voltage(C_I). \]
In fact, we only need this fact in the special case where $j < I$ (i.e., $j < i$, for all $i \in I$), so let us prove only this special case. (The general case uses the fact that $N$ is abelian, because it is a cyclic group, but we will not rely on this fact.)
For definiteness, let us assume $s_j = s$. (The other case is similar.) 
Letting $C_I = (s_1',s_2', \ldots, s_n')$, we have
	\begin{align*}
	\voltage(C_{I \cup \{j\}})
	&= s_1' s_2' \cdots s_j' a s_{j+1}' \cdots s_n'
	\\&= s_1' s_2' \cdots s_j' a  (s_1' s_2' \cdots s_j')^{-1} \cdot (s_1' s_2'  \cdots s_n') 
	\\&= s_1 s_2 \cdots s_j  a (s_1 s_2 \cdots s_j)^{-1} \cdot (s_1' s_2'  \cdots s_n') 
		&& \text{($j < j$)}
	\\&= a_j \cdot \voltage(C_I)
	. \end{align*}
This completes the proof of the claim.

Letting $x = \voltage(C)$, repeated application of the claim implies
	$ \voltage(C_I) = x \prod_{i \in I} a_i $.
By identifying $N$ with $\ZZ_{pq}$ in the natural way, we can now conclude from \cref{add3} that there is a subset~$I$ of~$J$, such that $\voltage(C_I)$ generates~$N$. So \cref{FGL} provides a hamiltonian cycle in $\Cay(G;S)$.
\end{proof}

\subsection{Some facts from group theory} \label{GrpThySect}

\begin{prop}[Hall's Theorem on solvable groups {\cite[Thm.~9.3.1(1), p.~141]{Hall-ThyOfGrps}}]\label{hal}
If $G$ is a solvable group of order $mn$, and $\gcd(m,n)=1$, then $G$ has at least one subgroup of order~$m$.
\end{prop}

\begin{lem}[well known] \label{sol}
If $p,q > 5$ are distinct primes, then every group of order $8pq$ is solvable.
\end{lem}

\begin{proof}
Equivalently, we wish to show that no divisor of $8pq$ is the order of a nonabelian simple group. Burnside's $2$-prime theorem \cite[Thm.~9.3.2, p.~143]{Hall-ThyOfGrps} tells us that the order of every nonabelian finite simple group is divisible by at least three distinct primes, so it suffices to show that the order of a simple group cannot be $2pq$, $4pq$, or~$8pq$. Here are two different ways to establish this.

\emph{First proof.}
 It was proved by J.\,G.\,Thompson \cite[Cor.~4, p.~388]{Thompson} that if the order of a simple group~$G$ is divisible by precisely three distinct primes, then $|G|$ is divisible by $2$ and~$3$ (and one other prime).
(In fact, there are only eight nonabelian simple groups whose order is divisible by only three distinct primes, and they are listed in \cite[Table~I, p.~3]{HuppertLempken}.)
 Since $3 \notin \{p,q\}$, this implies that $|G|$ is not a divisor of $8pq$.

\emph{Second proof.}
The conclusion can easily be derived from facts that appear in standard textbooks in group theory.
Suppose $G$ is a simple group of order $2^m pq$, with $m \in \{1,2,3\}$.

Let $P$ be a Sylow $p$-subgroup of~$G$.
We know from Sylow's Theorem that 
	\begin{itemize}
	\item $|G : N_G(P)| \equiv 1 \pmod{p}$,
	and 
	\item $|G : N_G(P)|$ is a divisor of $|G : P| = 2^m q$. 
	\end{itemize}
Furthermore, $P$ is abelian (indeed, it is cyclic of prime order), and Sylow subgroups of a nonabelian simple group cannot be in the centre of their normalizer \cite[Thm.~14.3.1, p.~203]{Hall-ThyOfGrps}, so $N_G(P) \neq P$, which means $|G : N_G(P)| \neq 2^m q$.

Suppose $|G : N_G(P)| = 8$. (This will lead to a contradiction.) Since $p > 5$ and $|G : N_G(P)| \equiv 1 \pmod{p}$, this implies $p = 7$. Also, since $|G : N_G(P)| = 8$, we have $|N_G(P)| = pq$. However, since $p,q > 5$ are distinct primes,
and $p = 7$, we know that $q > p$. Therefore, any group of order~$pq$ that has a normal subgroup of order~$p$ must be abelian; thus, we conclude that $N_G(P)$ is abelian. This contradicts the above-mentioned fact that Sylow subgroups of a nonabelian simple group cannot be in the centre of their normalizer.

We can now conclude that 
	\[ \text{$2^i q \equiv 1 \pmod{p}$, 
	\ for some $i \in \{0,1,2\}$.} \]
By the same argument, we also have $2^j p \equiv 1 \pmod{q}$, for some $j \in \{0,1,2\}$. So we see from \cref{0modpandq} that $\min(p,q) \le 5$, which contradicts the assumption that $p,q > 5$.
\end{proof}

\begin{rem}
The hypothesis that $p,q > 5$ in \cref{sol} can be weakened, because the first proof only requires $p,q > 3$. However, \cref{smallprime} easily handles the case where one of the primes is~$5$, so this strengthening of \cref{sol} would not shorten the proof of our main theorem.
\end{rem}

We will use the following fact in the proof of \cref{sinG'}, and also in \cref{AssumpSect}.

\begin{lem}[cf.~{\cite[Thm.~9.4.2, p.~146]{Hall-ThyOfGrps}}] \label{SquareFreeG'IsCyclic}
If $G$ is a finite group, and $|G'|$ is square-free, then $G'$ is cyclic.
\end{lem}

\Cref{GIsSemiprod} will place restrictive conditions on~$G$. We conclude our discussion of group theory 
with a two-part elementary observation about such groups:

\begin{lem} \label{sinG'}
Assume 
	\begin{itemize}
	\item $G$ is a group of order~$8pq$, where $p$ and~$q$ are distinct odd primes,
	and
	\item $G = P_2 \ltimes \cyclic{pq}$, where $|P_2| = 8$ and $\cyclic{pq}$ is a cyclic, normal subgroup of order~$pq$ that is contained in~$G'$.
	\end{itemize}
Then:
	\begin{enumerate}
	\item \label{sinG'-noZ}
	$\cyclic{pq} \cap Z(G)$ is trivial,
	and
	\item \label{sinG'-ham}
	if $S$ is a generating set of~$G$, such that $S \cap G' \neq \emptyset$, then $\Cay(G; S)$ has a hamiltonian cycle.
	\end{enumerate}
\end{lem}

\begin{proof}
\pref{sinG'-noZ}
This is a standard fact about relatively prime actions, but we provide a short proof.
Suppose $\cyclic{pq} \cap Z(G)$ is nontrivial. We may write $\cyclic{pq} = \cyclic{p} \times \cyclic{q}$ (uniquely) where $\cyclic{p}$ and~$\cyclic{q}$ are cyclic subgroups of order~$p$ and~$q$, respectively. For definiteness, let us assume that the subgroup~$\cyclic{q}$ is contained in $Z(G)$. Let $\widehat{G} = G/\cyclic{p} \cong P_2 \times \cyclic{q}$. It is obvious from this direct-product decomposition that $\cyclic{q} \nsubseteq G'$, which contradicts the assumption that $\cyclic{pq}$ is contained in~$G'$. 

\pref{sinG'-ham}
Let $s \in S \cap G' $. Since $|P_2| = 8$, we know that $|P_2'| \in \{1,2\}$. Therefore $|G'|$ is a divisor of $2pq$, so $|G'|$ is square-free, which implies that $G'$ is cyclic \csee{SquareFreeG'IsCyclic}. Then $\langle s \rangle$ is a subgroup of a cyclic, normal subgroup, so it is normal. 

Also note that $|\langle s \rangle|$, like $|G'|$, must be a divisor of $2pq$. This immediately implies that $|G : \langle s \rangle|$ is even (in fact, it is a multiple of~$4$). By~\pref{sinG'-noZ}, it also implies that $|\langle s \rangle \cap Z(G)| \le 2$. Therefore, $|\langle s \rangle \cap Z(G)|$ is a divisor of $|G : \langle s \rangle|$, so \cref{Z(G)factor} applies.
\end{proof}

\section{Using a computer to find hamiltonian cycles} \label{ComputerSect}

\subsection{Using a computer to apply the Factor Group Lemma} \label{ComputerFGL}

In the proof of the main theorem \pref{8pqThm}, we are given a group~$G$ of order $8pq$, and we wish to show that $\Cay(G; S)$ has a hamiltonian cycle, for every minimal generating set~$S$ of~$G$. This is accomplished by an extensive case-by-case analysis. However, as in \cite{Morris-NonSolvable,MorrisWilk}, we will use a computer to do the vast majority of the work. 

\begin{rem} \label{WeUseGAP}
Our computer programs are written in \textsf{GAP} \cite{GAP}.
The source code is available online at
	\[ \text{\url{https://arxiv.org/src/2304.03348/anc/}} \]
but this code relies on some of the programs of Morris-Wilk \cite{MorrisWilk} that are available at 
	\[ \text{\url{https://arxiv.org/src/1805.00149/anc/}} \]
So a reader who wishes to reproduce our results should combine all of the \textsf{.gap} files from both locations into a single directory.
\end{rem}

In most cases, the group $G$ is a semidirect product. More precisely, $G = \overline{G} \ltimes (\cyclic{p} \times \cyclic{q})$, where $\overline{G}$ is a subgroup of order~$8$, and $\cyclic{p}$ and~$\cyclic{q}$ are cyclic, normal subgroups of order~$p$ and~$q$, respectively \csee{GIsSemiprod}. The quotient $G/(\cyclic{p} \times \cyclic{q})$ can be naturally identified with~$\overline{G}$. By \cref{FGL}, it suffices to find a hamiltonian cycle~$C$ in $\Cay(\overline{G}; S)$ whose voltage generates $\cyclic{p} \times \cyclic{q}$. The graph $\Cay(\overline{G}; S)$ has only $8$ vertices, so it is easy to have a computer find all of its hamiltonian cycles, calculate their voltages, and determine whether there is a good one.

The key difficulty is that there are infinitely many possibilities for the primes~$p$ and~$q$, but a computer can only do finitely many calculations. A method that addresses this issue can be found in \cite[Lem.~3.3]{MorrisWilk}. The idea is to let $Z$ be the subring of~$\CC$ that is generated by the roots of unity and let $\mu$ be the group of all roots of unity. 
Any semidirect product $\overline{G} \ltimes (\cyclic{p} \times \cyclic{q})$ arises from a pair of twist homomorphisms $\zeta_p \colon \overline{G} \to \Aut \cyclic{p}$ and $\zeta_q \colon \overline{G} \to \Aut \cyclic{q}$. Since $\Aut \cyclic{p}$ and $\Aut \cyclic{q}$ are cyclic, they can be identified with subgroups of~$\mu$. Therefore, $\zeta_p$ and~$\zeta_q$ correspond to homomorphisms $\widehat \zeta_p \colon \overline{G} \to \mu$ and $\widehat \zeta_q \colon \overline{G} \to \mu$. After constructing the corresponding semidirect products $G_p = \overline{G} \ltimes_{\widehat \zeta_p} Z$ and $G_q = \overline{G} \ltimes_{\widehat \zeta_q} Z$, a computer program can calculate the voltages~$\pi_p$ and~$\pi_q$ of any hamiltonian cycle~$C$ in both of these groups. These voltages are algebraic integers, so they have a ``norm\rlap,''  which is an element of~$\ZZ$. It is not difficult to see that if $\Norm \pi_p \not\equiv 0 \pmod{p}$ and $\Norm \pi_q \not\equiv 0 \pmod{q}$, then $\voltage(C)$ generates $\cyclic{p} \times \cyclic{q}$ (see the proof of \cite[Lem.~3.3]{MorrisWilk}, with $n = 1$). Therefore, it suffices to show, for every pair of distinct primes $p$ and~$q$, that there exists a hamiltonian cycle in $\Cay(\overline{G}; S)$, such that $\lcm(\Norm \pi_p, \Norm \pi_q)$ is relatively prime to~$pq$.
Actually, by \cref{pq>5}, we will only need to consider primes that are greater than~$5$.

Here is a bit more explanation of how the computer programs work.
Write $S = \{s_1,\ldots,s_k\}$, and fix generators $x_p$ and $x_q$ of $\cyclic{p}$ and~$\cyclic{q}$, respectively. Each element~$s$ of~$S$ can be written uniquely in the form $\overline{s} \, x_p^i \,x_q^j$, with $\overline{s} \in \overline{G}$. Let us say that $s$ \emph{involves}~$x_p$ if $x_p^i \neq 0$; similarly, $s$ \emph{involves}~$x_q$ if $x_q^j \neq 0$.

\refstepcounter{caseholder}

\begin{case} \label{ComputerFGL-OnlyOne}
The simplest case for computation is when we know that only one element~$s_m$ involves~$x_p$, and only one element~$s_n$ involves~$x_q$. 
\end{case}
(It is possible that $m = n$.) Every nontrivial element of a cyclic group of prime order is a generator, so we may assume $s_m = \overline{s_m}\, x_p$ and $s_n = \overline{s_n} \, x_q$ (unless $m = n$, in which case we have $s_m = s_n = \overline{s_m} x_p x_q$). To consider all possibilities, we have the computer:
	\begin{itemize}
	\item loop through all groups~$\overline{G}$ of order~$8$,
	\item loop through all generating sets~$\overline{S}$ of~$\overline{G}$,
	\item loop through all hamiltonian cycles $\overline{C}$ in $\Cay(\overline{G}; \overline{S})$,
	and
	\item loop through all homomorphisms $\zeta_p \colon \overline{G} \to \mu$ and  $\zeta_q \colon \overline{G} \to \mu$ (these homomorphisms are called \emph{abelian characters} of~$\overline{G}$).
	\end{itemize}

Actually, the program must allow $\overline{S}$ to be a multiset, because two different elements of~$S$ may have the same image in~$\overline{G}$. (However, we will see in \cref{Svsd(Gbar)} that the cardinality of~$S$ is at most~$5$, so this is still a finite problem.) Then each hamiltonian cycle~$\overline{C}$ in $\Cay(\overline{G}; \overline{S})$ may have many different possible lifts to a walk in $\Cay(G; S)$. We refer to these walks as ``coded'' hamiltonian cycles, because we encode each walk as a sequence of numbers, by making a list of the elements of $S \cup S^{-1}$, and specifying each edge of the walk by recording the index of the corresponding element of the list.

Now, for each coded hamiltonian cycle~$C$ in $\Cay(\overline{G}; S)$, the computer calculates the voltages $\pi_p$ and~$\pi_q$ of the corresponding walks in $\Cay(\overline{G} \ltimes_{\widehat\zeta_p} Z; S_p)$ and $\Cay(\overline{G} \ltimes_{\widehat\zeta_q} Z; S_q)$, where 
	\begin{itemize}
	\item $S_p$ is obtained from~$\overline{S}$ by replacing $s_m$ with~$s_m x_p$,
	and
	\item $S_q$ is obtained from~$\overline{S}$ by replacing $s_n$ with~$s_n x_q$.
	\end{itemize}
(Here, $x_p$ and~$x_q$ are represented by the element $(0,1)$ of $\overline{G} \ltimes_{\widehat\zeta_p} Z$ or $\overline{G} \ltimes_{\widehat\zeta_q} Z$. However, in order to be consistent with the conventions used in the Morris-Wilk programs, the order of the factors needs to be reversed: to be precise, the programs compute in the groups $Z \rtimes_{\zeta_p} \overline{G}$ and $Z \rtimes_{\zeta_q} \overline{G}$, so $x_p$ and~$x_q$ are actually represented by the element $(1,0)$.)

If the computer finds a hamiltonian cycle, such that $\lcm \bigl( \Norm(\pi_p) , \Norm(\pi_q) \bigr)$ has no prime divisors greater than~$5$, then we know that $\Cay(G;S)$ has a hamiltonian cycle for all $p$ and~$q$ (greater than~$5$), so the computer can move on to the next iteration of the loop. On the other hand, the program will raise an error if there is no such hamiltonian cycle. It is important to note that this never happens in our calculations, because all cases where the computer search would fail are handled separately (see \cref{ElemAbelSect,SpecialDihedralSect}). 

\begin{case} \label{ComputerFGL-TwoInvolved}
The situation is more complicated when $x_p$ or~$x_q$ may be involved in more than one element of~$S$.
\end{case}
In all cases of the proof, we are able to use theoretical arguments to reduce to a situation where $x_p$ and~$x_q$ are not \emph{both} involved in more than one element of~$S$. Therefore, let us assume that 
	\begin{itemize}
	\item $x_p$ is involved in only one element of~$S$, 
	but 
	\item $x_q$ is involved in~$s_1$, and may (or may not) be involved in~$s_2$ (and is certainly not involved in any other element of~$S$).
	\end{itemize}

For each coded hamiltonian cycle~$C$, we calculate the voltage $\pi_p$ of~$C$ in $\overline{G} \ltimes_{\widehat\zeta_p} Z$, exactly as in \cref{ComputerFGL-OnlyOne}. If no prime divisor of $\Norm \pi_p$ is greater than~$5$, then we know that the voltage of~$C$ generates~$\cyclic{p}$, so checking whether the voltage generates~$\cyclic{q}$ is the only remaining issue. (On the other hand, if some prime divisor is greater than~$5$, then we discard this hamiltonian cycle as being useless.)

We deal with the prime~$q$ by a different approach that was introduced in \cite[Lem.~3.3]{MorrisWilk}. Namely, we calculate the voltage of~$C$ in~$\overline{G} \ltimes_{\widehat\zeta_q} Z$ with respect to two different connection sets $S_q' = \{s_1', \ldots, s_k'\}$ and $S_q''= \{s_1'', \ldots, s_k''\}$. In $S_q'$, the generator $s_1'$ is the only element that involves~$x_q$; in $S_q''$, it is~$s_2''$ that involves~$x_q$. Let us use $\pi_q'(C)$ and $\pi_q''(C)$ to denote the corresponding voltages (in $\overline{G} \ltimes_{\widehat\zeta_q} Z$). 

A key observation in the proof of \cite[Lem.~3.3]{MorrisWilk} is that there is a homomorphism $\Phi \colon Z \to \cyclic{q}$, such that if we write $s_1 = \overline{s_1} x_q$ and $s_2 = \overline{s_2} x_q^i$ (modulo~$\cyclic{p}$), and let $\pi_q(C)$ be the voltage of~$C$ in $\overline{G} \ltimes_{\zeta_q} \cyclic{q}$, then
	\[ \pi_q(C) = \Phi \bigl( \pi_q'(C) \bigr) +  i \, \Phi \bigl( \pi_q''(C) \bigr) . \]
Therefore, if it happens to be the case that $\pi_q''(C) = 0$, and it is also true that $\Norm \bigl( \pi_q'(C) \bigr)$ does not have any prime divisors greater than~$5$, then $\pi_q(C)$ generates~$\cyclic{q}$. Since we already know from above that $\pi_p(C)$ generates~$\cyclic{p}$, this implies that $\Cay(G;S)$ has a hamiltonian cycle. So we can pass to the next iteration of the loop. The program refers to this as finding a ``single'' hamiltonian cycle.

Another key observation in the proof of \cite[Lem.~3.3]{MorrisWilk} follows from undergraduate-level linear algebra: if $C_1$ and~$C_2$ are two coded hamiltonian cycles, and the norm of
	\[ \det \begin{bmatrix} \pi_q'(C_1) & \pi_q''(C_1) \\ \pi_q'(C_2) & \pi_q''(C_2) \end{bmatrix} \]
is not divisible by~$q$, then $C_1$ and~$C_2$ cannot both have trivial voltage in $\overline{G} \ltimes_{\widehat\zeta_q} Z$. Hence, \cref{FGL} applies to at least one of them, so $\Cay(G;S)$ has a hamiltonian cycle.
Therefore, when a ``single'' is not found, the program searches through all pairs of hamiltonian cycles $C_1$ and~$C_2$, to find a case where the norm of the determinant is not divisible by any prime greater than~$5$.

\penalty-1000 

\begin{rem}
\leavevmode
\noprelistbreak 
	\begin{enumerate}
	\item We said above that the computer loops through all groups, all generating sets, and all abelian characters, but that is not actually true. Slightly different computer programs were written for different cases of the proof, and each case  puts restrictions on the groups, generating sets, or abelian characters that need to be considered. 
	\item The programs in 
	\textsf{8pq-Prop-4-1.gap}, \textsf{8pq-Prop-7-4.gap}, and \textsf{8pq-Prop-7-7.gap}
use the method of \cref{ComputerFGL-OnlyOne}, but the program in \textsf{8pq-Prop-7-9.gap} deals with \cref{ComputerFGL-TwoInvolved}.
	\end{enumerate}
\end{rem}

\mathversion{bold}
\subsection{An anomalous case with \texorpdfstring{$q = 7$}{q=7}} \label{56pSect}
The programs described in \cref{ComputerFGL} assume that 
\mathversion{normal}%
$G$ is a semidirect product $\overline{G} \ltimes (\cyclic{p} \times \cyclic{q})$, but this is not always the case. In this \lcnamecref{56pSect}, we deal with a situation where that assumption is not true, by using some of the computer programs that accompany the Morris-Wilk paper~\cite{MorrisWilk}. As was already mentioned in \cref{WeUseGAP}, these programs are online at
	\[ \text{\url{https://arxiv.org/src/1805.00149/anc/}} \]
The particular programs used in this \lcnamecref{56pSect} make extensive use of K.\,Helsgaun's program \textsf{LKH}~\cite{LKH}, which implements a very effective heuristic for finding hamiltonian cycles. (So \textsf{LKH} must also be installed.)

\begin{lem}[cf.\ {\cite[Rem.~1.4(4)]{MorrisWilk}}] \label{56HamConn}
If $H$ is the unique nonabelian semidirect product of the form $\cyclic{7} \ltimes (\cyclic{2})^3$, then every connected Cayley graph on~$G$ is hamiltonian connected.
\end{lem}

\begin{proof}
The Morris-Wilk program \textsf{1-3-HamConnOrLaceable.gap} establishes that every connected Cayley graph of order less than~$64$ (and valency at least~$3$) is either hamiltonian connected or hamiltonian laceable. However, it takes a long time to run, and the official report in \cite[Prop.~1.3]{MorrisWilk} only states the result for orders less than~$48$. 

The program loops through all orders from 3 to~63, and loops through all groups of each order. 
(See \cite[\S2C]{MorrisWilk} for more explanation.)
To quickly prove the case we need, change two lines in the program:
	\begin{itemize}
	\item change \quad \verb|for k in [3..63] do| 
		\\\hphantom{change \quad}\hbox to 0pt{\hss to \quad }\verb|for k in [56] do|
	\item change \quad \verb|for GapId in [1..NumberSmallGroups(k)] do| \quad
		 \\\hphantom{change \quad}\hbox to 0pt{\hss to \quad }\verb|for GapId in [11] do|
	\end{itemize}
This modified program will print
\\\hbox{}\hfil \verb|G = SmallGroup(56,11) = (C2 x C2 x C2) : C7| \\
which confirms that the correct group is being considered, then will print a few lines of progress reports, followed by a  statement that all of the Cayley graphs are hamiltonian connected or hamiltonian laceable. 

However, the group~$H$ has no subgroup of index~$2$, so none of its connected Cayley graphs are bipartite; therefore, none of its connected Cayley graphs are hamiltonian laceable. Hence, all of them must be hamiltonian connected.
\end{proof}

\begin{prop} \label{56pBad}
Let $H$ be the unique nonabelian semidirect product of the form $\cyclic{7} \ltimes (\cyclic{2})^3$. If $G = H \ltimes \cyclic{p}$, for some prime $p > 5$, then every connected Cayley graph on~$G$ has a hamiltonian cycle.
\end{prop}

\begin{proof}
In~\cite{MorrisWilk}, it is proved that every connected Cayley graph of order $kp$ is hamiltonian when $1 \le k < 48$ (and $p$~is prime). 
In the current situation, we have $k = 56$, but it is easy to adapt the argument. Actually, we do not need the entire argument, just two short parts of it.

Let $S$ be a generating set of~$G$; we wish to show $\Cay(G;S)$ has a hamiltonian cycle. By \cref{OnlyMinimal}, we may assume that $S$ is irredundant. Let $\overline G = G/\cyclic{p} \cong H$. 

\refstepcounter{caseholder} 

\begin{case}
Assume $\overline S$ is a redundant generating set of~$\overline G$.
\end{case}
This case follows from the proof of \cite[Lem.~4.2]{MorrisWilk}. For the reader's convenience, we sketch the argument.

Choose a (proper) subset~$S_0$ of~$S$, such that $\overline{S_0}$ is an irredundant generating set of~$\overline{G}$. Since $\overline{S_0}$ generates $\overline{G}$, we know that $|\langle S_0 \rangle|$ is divisible by $|\overline{G}| = 56$. However, we also know $\langle S_0 \rangle \neq G$, since the generating set~$S$ is irredundant. We conclude that $|\langle S_0 \rangle| = 56$, so, after passing to a conjugate, $\langle S_0 \rangle = H$.

Since $|G| / |\overline{G} = p$ is prime, it is easy to see that $|S| = |S_0| + 1$; hence, we have $S = S_0 \cup \{a\}$ for some $a \in S$. By \fullcref{5}{prime}, we may assume $\overline{a}$ is nontrivial. 
Therefore \cref{56HamConn} provides a hamiltonian path $(\overline{s_1}, \ldots, \overline{s_{55}})$ from~$\overline{1}$ to~$\overline{a^{-1}}$ in $\Cay(\overline{G}; \overline{S_0})$.  Then
	\[ C = (\overline{s_1}, \ldots, \overline{s_{55}}, \overline{a}) \]
is a hamiltonian cycle in $\Cay(\overline{G}; \overline{S})$. 

Write $a = hz$ with $h \in H$ and $z \in \cyclic{p}$. Since $\langle S_0, h \rangle = H \neq G$, it must be the case that $z$ is nontrivial. Since $\cyclic{p}$ has prime order, this implies that $z$ generates~$\cyclic{p}$.
The voltage of the hamiltonian cycle is
	\[ \voltage(C) = s_1 s_2 \cdots s_{55} a =  s_1 s_2 \cdots s_{55} h z . \]
However, since 
	\[ \overline{s_1 s_2 \cdots s_{55} h} = \overline{s_1 s_2 \cdots s_{55} a} = \overline{1} \]
and $s_1 s_2 \cdots s_{55} h \in H$, we must have $s_1 s_2 \cdots s_{55} h = 1$. Therefore $\voltage(C) = z$ generates~$\cyclic{p}$. So \cref{FGL} applies.

\begin{case}
Assume $\overline S$ is an irredundant generating set of~$\overline G$.
\end{case}
For every nontrivial group~$H$ of order less than~$48$, the Morris-Wilk program \textsf{3-4-IrredundantSBar.gap} verifies that if 
	\begin{itemize}
	\item $p$ is a prime number, 
	\item $G$ is any semidirect product $H \ltimes \cyclic{p}$,
	and
	\item $S$ is any irredundant generating set of~$G$, such that the projection of~$S$ to~$H$ is an irredundant generating set,
	\end{itemize}
then $\Cay(G;S)$ has a hamiltonian cycle. It does this by looping through all orders from 1 to~47, then looping through all possible groups of each order.

This computer program can easily be modified to consider the case here. 
Instead of looping through all groups of many different orders, we just want to look at a single group of order~$56$.
As in the proof of \cref{56HamConn}, it suffices to change two lines in the program. Specifically:
	\begin{itemize}
	\item change \quad \verb|for k in [1..47] do| 
		\\\hphantom{change \quad}\hbox to 0pt{\hss to \quad }\verb|for k in [56] do|
	\item change \quad \verb|for GapId in [1..NumberSmallGroups(k)] do| \quad
		 \\\hphantom{change \quad}\hbox to 0pt{\hss to \quad }\verb|for GapId in [11] do|
	\end{itemize}
The modified program should take less than a minute to run. Since it completes successfully, rather than raising an error, we conclude that $\Cay(G;S)$ has a hamiltonian cycle.
\end{proof}

\section{Assumptions and notation} \label{AssumpSect}

This short \lcnamecref{AssumpSect} establishes that we may make some simplifying assumptions when proving the main theorem \pref{8pqThm}. All later sections will make use of the assumptions and notation that are introduced here.

\begin{notation}
Let $G$ be a finite group, such that 
	\[ \text{$|G|=8pq$ where $p$ and~$q$ are distinct prime numbers}, \] 
and let $S$ be a generating set of~$G$.
\end{notation}

To prove \cref{8pqThm}, we wish to show that $\Cay(G;S)$ has a hamiltonian cycle. 
By \cref{OnlyMinimal}, the following causes no loss of generality:

\begin{assump}
The generating set~$S$ is irredundant.
\end{assump}

We first consider the case where (at least) one of the primes is small:

\begin{lem} \label{smallprime}
If\/ $\min(p,q) \le 5$, then every connected Cayley graph on~$G$ has a hamiltonian cycle.
\end{lem}

\begin{proof}
Assume, without loss of generality, that $\min(p,q) = q$, so $q \le 5$. 
Then $8q < 48$, so \fullcref{CertainOrders}{kp} applies with $k = 8q$. 
\end{proof}

In view of this \lcnamecref{smallprime}, we henceforth make the following assumption:

\begin{assump} \label{pq>5}
$\min(p, q) > 5$.
\end{assump}

Now, the following result is an easy (but crucial!) consequence of \cref{sol}.

\begin{prop} \label{semiprod}
One of the following is true: either
	\begin{enumerate}
	\item \label{semiprod-semiprod}
	$G = P_2 \ltimes H$, where $P_2$ is a Sylow $2$-subgroup \textup(so $|P_2| = 8$\textup) and $H$ is a normal subgroup of order~$pq$,
	or
	\item the assumptions of \cref{56pBad} are satisfied \textup(perhaps after interchanging $p$ and~$q$\textup), so every connected Cayley graph on~$G$ has a hamiltonian cycle.
	\end{enumerate}
\end{prop}

\begin{proof}
This is a standard argument.
Let $P_2$ be a Sylow $2$-subgroup of~$G$. The group $G$ is solvable \csee{sol}, so \cref{hal} tells us there is a subgroup~$H$ of order~$pq$. Assume, without loss of generality, that $p > q$, and let $P$ be a Sylow $p$-subgroup of~$H$. Then it is easy to see from Sylow's Theorem (and is well known \cite[p.~49]{Hall-ThyOfGrps}) that $P$ is a normal subgroup of~$H$, so $H \subseteq N_G(P)$. Hence, if we let $n_p$ be the number of Sylow $p$-subgroups of~$G$, then (by Sylow's Theorem) we see that
	\[ \text{$n_p = |G : N_G(P)| \text{ is a divisor of } |G : H| = 8$,
	so $n_p \in \{1, 2, 4, 8\}$.} \]
However, we also know from Sylow's Theorem that $n_p \equiv 1 \pmod{p}$. Since $p > q > 5$, we have $p > 7$, so we can conclude that $n_p = 1$. This means that $P$ is the unique Sylow $p$-subgroup of~$G$, so $P$ is a normal subgroup of~$G$.

Now $H/P$ is a Sylow $q$-subgroup of $G/P$. If $H/P \triangleleft G/P$, then $H \triangleleft G$, so conclusion~\pref{semiprod-semiprod} holds.

We may therefore assume $H/P$ is not normal, so $G/P$ has more than one Sylow $q$-subgroup. By \cref{hal}, we may let $K$ be a subgroup of order~$8q$ in~$G$, so $G = K \ltimes P$. Then $K \cong G/P$ has more than one Sylow $q$-subgroup. Thus, if we let $\cyclic{q}$ be a Sylow $q$-subgroup of~$K$, then $|K : N_K(\cyclic{q})| > 1$. Since $|K : N_K(\cyclic{q})| \le |K: \cyclic{q}| = 8$ (and we know $|K : N_K(\cyclic{q})| \equiv 1 \pmod{q}$ by Sylow's Theorem), we conclude that $q = 7$ and $N_K(\cyclic{q}) = \cyclic{q}$. So $K$ has a normal $q$-complement \cite[Thm.~14.3.1, p.~203]{Hall-ThyOfGrps}: $K = \cyclic{q} \ltimes P_2$ (after replacing $P_2$ by a conjugate, so it is contained in~$K$). It is not difficult to see that $\cyclic{7} \ltimes (\cyclic{2})^3$ is the only semidirect product of the form $\cyclic{7} \ltimes P_2$, such that $P_2$ has order~$8$ and is not centralized by~$\cyclic{7}$. So $G$ is as described in \cref{56pBad}.
\end{proof}

We may assume it is the condition in part~\pref{semiprod-semiprod} of the \lcnamecref{semiprod} that is satisfied.
Now, $|G/H| = 8$, so $|(G/H)'| \in \{1,2\}$. Therefore 
	\begin{align} \label{G'divides}
	\text{$|G'|$ is a divisor of $2pq$} 
	. \end{align}

We may assume $|G'|$ is divisible by~$pq$, for otherwise $|G'|$ is either~$1$ or prime or twice an odd prime, so \cref{4} applies.
This implies $H \subseteq G'$. On the other hand, \pref{G'divides} implies that $|G'|$ is square-free. So $G'$ is cyclic \csee{SquareFreeG'IsCyclic}. Since subgroups of cyclic groups are cyclic, we conclude that $H$ is cyclic. Hence, the following condition is satisfied:

\begin{assump} \label{GIsSemiprod}
We have~
	\[ G = P_2 \ltimes \cyclic{pq} , \]
where $|P_2| = 8$, and $\cyclic{pq}$ is a cyclic, normal subgroup of order~$pq$ that is contained in~$G'$.
\end{assump}

\begin{notation} 
Let:
	\begin{itemize}
	\item $\overline G = G / \cyclic{pq} \cong P_2$,
	\item $\cyclic p$ be the subgroup of~$\cyclic{pq}$ that has order~$p$,
	\item $\cyclic q$ be the subgroup of~$\cyclic{pq}$ that has order~$q$,
	\item $x_p$ be a generator of~$\cyclic p$,
	and
	\item $x_q$ be a generator of~$\cyclic q$.
	\end{itemize}
Then 
	\[ \cyclic{pq} = \cyclic p \times \cyclic q = \langle x_p \rangle \times \langle x_q \rangle . \]
\end{notation}

\section{Some cases where \texorpdfstring{$\overline G \cong (\cyclic2 )^3$}{the quotient is elementary abelian}} \label{ElemAbelSect}

\begin{prop} \label{ElemAbel}
The assumptions and notation of \cref{AssumpSect} are in effect.
Also assume $\overline G \cong (\cyclic2 )^3$, and  either
	\begin{enumerate}
	\item \label{ElemAbel-S=3}
	$|S| = 3$, 
	or
	\item \label{ElemAbel-not8}
	$|S| = 4$ and
	there does not exist a subset~$S_0$ of~$S$, such that $|\langle S_0 \rangle | = 8$.
	\end{enumerate}
Then every connected Cayley graph on~$G$ has a hamiltonian cycle.
\end{prop}

\refstepcounter{caseholder} 

\begin{proof}
For convenience, let us recall some terminology from \cref{ComputerFGL} for use in this proof. Every element~$g$ of~$G$ can be written in the form $\overline{g} x_p^i x_q^j$, where $\overline{g} \in P_2$ and $i,j \in \ZZ$. If $x_p^i$ is nontrivial, we say that $g$ \emph{involves}~$x_p$; similarly, if $x_q^j$ is nontrivial, we say that $g$ \emph{involves}~$x_q$.

Now, we consider each of the two possibilities for~$|S|$ as a separate case.

\begin{case}
Assume  $|S|=3$. 
\end{case}
Write $S = \{a,b,c\}$. We have the following hamiltonian cycle in $\Cay(G/\cyclic{pq}; S)$: 
	\[ C_{a,b,c} = (a^{-1}, b^{-1}, a, c^{-1}, a^{-1}, b, a, c) . \]
Its voltage is
	\[ \voltage(C_{a,b,c}) 
	= a^{-1} b^{-1} a c^{-1} a^{-1} b a c
	= (b^a)^{-1} c^{-1} b^a c
	= [b^a, c] , \]
where $b^a = a^{-1} b a$ denotes the \emph{conjugate} of~$b$ by~$a$.

\begin{subcase} \label{ElemAbel3-aCent} \label{ElemAbelPf-centralize}
Assume some element of~$S$ centralizes~$\cyclic{pq}$. 
\end{subcase}
For definiteness, assume that $a$ centralizes~$\cyclic{pq}$. 
	\begin{itemize}
	\item If $|a| = 2$, then $a \in Z(G)$, so \fullcref{5}{Z} applies with $s = a$. 
	\item If $|a| \in \{2p, 2q\}$, then \cref{n1} applies with $s = a$, $t = s^{-1}$, and $N = \langle a^2 \rangle$.
	\end{itemize}
So we may assume $|a| = 2pq$. The hamiltonian cycle $C_{a,b,c}$ has 4~occurrences of~$a$ or~$a^{-1}$. (In fact, 3~occurrences would be enough.) Therefore, we see from \cref{occur3} that there is a hamiltonian cycle in $\Cay(G;S)$.

\begin{subcase} \label{ElemAbelPf-invert}
Assume every element of~$S$ inverts~$\cyclic{pq}$. 
\end{subcase}
This implies that $G \cong (\cyclic{2}) ^2\times D_{2pq}$ is a generalized dihedral group (with $A = (\cyclic2 )^2 \times \cyclic{pq}$), so a hamiltonian cycle is provided by \cref{DihTypeDivBy4}.

\begin{subcase} \label{ElemAbelPf-2invert}
Assume two elements of~$S$ invert~$\cyclic{pq}$.
\end{subcase}
For definiteness, let us say that $b$ and~$c$ invert~$\cyclic{pq}$ (but $a$ does not, for otherwise \cref{ElemAbelPf-invert} applies). We may assume $a$ does not centralize all of~$\cyclic{pq}$ (for otherwise \cref{ElemAbelPf-centralize} applies), so $a$ inverts~$\cyclic{p}$ and centralizes~$\cyclic{q}$ (perhaps after interchanging $p$ and~$q$). Since $c$ has trivial centralizer in~$\cyclic{pq}$, we may conjugate by an element of~$\cyclic{pq}$ to assume $c$ is in $(\cyclic{2})^3$. (So $c$ does not involve~$x_p$ or~$x_q$.)
We may also assume that $a$ does not involve~$x_q$, for otherwise $|a| = 2q$, so \cref{n1} applies with $s = a$, $t = a^{-1}$, and $N = \cyclic{q}$.
Write $a = e_1 x_p^i$ and $b = e_2 x_p^j x_q$, with $e_1, e_2 \in (\cyclic{2})^3$. Note that $e_1$, like~$a$, inverts~$x_p$ and centralizes~$x_q$, whereas $e_2$ inverts both $x_p$ and~$x_q$. Therefore
	\[ b^a
	= a^{-1} b a
	= x_p^{-i} e_1  \cdot e_2 x_p^j x_q \cdot e_1 x_p^i 
	= e_2 x_p^{2i - j} x_q , \]
so 
	\[ \text{$\voltage(C_{a,b,c})  = [b^a, c]$ generates $\cyclic{pq}$ if and only if $2i \not\equiv j \pmod{p}$.} \]
Therefore, we may assume $i = 1$ and $j = 2$, so $a = e_1 x_p$ and $b = e_2 x_p^2 x_q$.

Let $\widehat G = G/\cyclic{p} \cong \cyclic2  \times D_{4q}$, where $\langle \widehat a \rangle = \cyclic2  \times \{1\}$ and $\langle \widehat b, \widehat c \rangle = \{1\} \times D_{4q}$. Then 
	\[ C_1 = \bigl( (b, c)^{2q} \#, a \bigr)^2 \] 
is a hamiltonian cycle in $\Cay( \widehat G ; a, b, c)$. Its voltage is
	\begin{align*}
	\voltage(C_1) 
	&= \bigl( (b c)^{2q} \, c a \bigr)^2
	\\&= \bigl( (e_2 x_p^2 x_q \cdot c)^{2q} \cdot c \cdot e_1 x_p \bigr)^2
	\\&= \bigl( (e_2 c \cdot  x_p^{-2} x_q^{-1} )^{2q} \cdot c e_1 \cdot x_p \bigr)^2
	\\&= \bigl( x_p^{-4q} \, x_p \bigr)^2 \bigl( c e_1 \bigr)^2 
	\\&= x_p^{2(1-4q)}
	. \end{align*}
Therefore, if $\voltage(C_1)$ is trivial, then 
	\[ 4q \equiv 1 \pmod{p} . \]

Now, let $\widecheck G = G/\cyclic{q} \cong \cyclic2  \times D_{4p}$. We claim that the following is a hamiltonian cycle in $\Cay( \widecheck G ; a, b, c)$:
	\[ C_2 = \bigl( (b, c)^{p} \#, a, (c, b)^{p} \#, a \bigr)^2 . \]
In fact, $\widecheck G$ is a generalized dihedral group, with $A = \ZZ_2 \times \ZZ_{2p}$, and $C_2$ is the hamiltonian cycle that is constructed in the proof of \cite[Cor.~2.3]{AlspachZhang} for this particular group. However, we provide a short proof for completeness. First, note that the length~$8p$ of this walk is correct for a hamiltonian cycle. Also note that the walk is closed, because, by using the fact that $(\widecheck b \widecheck c)^{p} = (\widecheck c \widecheck b)^{p}$ is an element of order~$2$ in the centre of~$\widecheck G$, we see that
	\[ (b c)^{p} c \cdot a \cdot (c b)^{p} b \cdot a
	\equiv c \cdot a \cdot b \cdot a
	\equiv c \cdot e_1 x_p \cdot e_2 x_p^2 \cdot e_1 x_p
	= e_2 c 
	\equiv (\widecheck b \widecheck c)^{p}
	\pmod{\cyclic{q}} \]
has order~$2$, modulo~$\cyclic{q}$. It therefore suffices to show that this cycle passes through all of the vertices of the Cayley graph. Let 
	\[ \mathcal{V} = \{\, (\widecheck b \, \widecheck c)^i \, \widecheck b^j \mid 0 \le i < p, \ 0 \le j \leq 1 \,\} 
		\text{\quad and\quad}
		\mathcal{W} = \{\, (\widecheck c \, \widecheck b)^i \, \widecheck c^j \mid 0 \le i < p, \ 0 \le j \leq 1 \,\} 
		. \]
Then $C_2$ passes through the vertices in
	\[ \mathcal{V} 
	\ \cup \ (\widecheck b \widecheck c)^p \, \widecheck c \, \widecheck a \, \mathcal{W}
	\ \cup \  (\widecheck b \, \widecheck c)^{p} \,  \mathcal{V} 
	\ \cup \  \widecheck c \, \widecheck a \, \mathcal{W}
	\ = \ \langle \widecheck b, \widecheck c \rangle 
	\ \cup \ \widecheck c \, \widecheck a  \langle \widecheck b, \widecheck c \rangle
	\ = \  \widecheck G . \]
This completes the proof of the claim.

The voltage of this hamiltonian cycle is
	\begin{align*}
	\voltage(C_2)
	&= \bigl( (b c)^{p} c \cdot a \cdot (c b)^{p} b \cdot a \bigr)^2
	\\&\equiv \bigl( (e_2 x_q c)^{p} c \cdot e_1 \cdot (c e_2 x_q)^{p} e_2 x_q \cdot e_1 \bigr)^2
		& \pmod{\cyclic{p}}
	\\&= \bigl( (e_2 c x_q^{-p}) c \cdot e_1 \cdot (ce_2 x_q^{p}) e_2 x_q \cdot e_1 \bigr)^2
	\\&= ( x_q^{1-2p} e_2 c )^2
	\\&= x_q^{2(1-2p)}
	. \end{align*}
If this voltage is trivial, then $2p \equiv 1 \pmod{q}$.

To complete the proof of this \lcnamecref{ElemAbelPf-2invert}, we show that this is impossible. Write $p = (kq + 1)/2$, for some $k \in \NN^+$. However, recall that we also know $4q \equiv 1 \pmod{p}$, so, for some $\ell \in \NN^+$, we have 
	\[ q 
	= \frac{\ell p + 1}{4} 
	= \frac{\displaystyle \ell  \cdot \frac{kq + 1}{2}+ 1}{4} 
	= \frac{k \ell q + \ell + 2}{8}
	.\]
This obviously implies $k \ell < 8$ and $\ell + 2 \equiv 0 \pmod{q}$. Since $q \ge 7$ \csee{pq>5}, we conclude that $q = 7$, $\ell = 5$, and $k = 1$. But then $(k \ell q + \ell + 2)/8 = 42/8$ is not an integer.
This is a contradiction.

\begin{subcase}
Assume that precisely one element of~$S$ inverts~$\cyclic{pq}$. 
\end{subcase}
We may assume it is $c$ that inverts~$\cyclic{pq}$ and (after conjugating by an element of~$\cyclic{pq}$) that $c$ does not involve $x_p$ or~$x_q$. Then $a$ and~$b$ each have a nontrivial centralizer in~$\cyclic{pq}$. 
Note that if $a$ and~$b$ both centralize~$\cyclic{p}$, then we can assume that neither of them involves $x_p$ (otherwise \cref{n1} applies with $s \in \{a,b\}$, $t = s^{-1}$, and $N = \cyclic{p}$), so no element of~$S$ involves~$x_p$, which contradicts the fact that $S$ generates~$G$. Similarly, we can assume that $a$ and~$b$ do not both centralize~$\cyclic{q}$. Hence, we may assume that $a$ centralizes~$\cyclic{q}$ and $b$~centralizes~$\cyclic{p}$. 

Then we have
	\[ a = e_1 x_p, \quad b = e_2 x_q, \quad c = e_3 , \]
where $\langle e_1, e_2, e_3 \rangle = (\cyclic2 )^3$, and:
	\begin{itemize}
	\item $e_1$ inverts~$\cyclic{p}$ and centralizes~$\cyclic{q}$,
	\item $e_2$ centralizes~$\cyclic{p}$ and inverts~$\cyclic{q}$,
	and
	\item $e_3$ inverts $\cyclic{p}$ and~$\cyclic{q}$.
	\end{itemize}
As in \cref{ElemAbelPf-2invert}, let $\widehat G = G/\cyclic{p} \cong \cyclic2  \times D_{4q}$, where $\langle \widehat a \rangle = \cyclic2  \times \{1\}$ and $\langle \widehat b, \widehat c \rangle = D_{4q}$. Then 
	\[ C_1 = \bigl( (b, c)^{2q} \#, a \bigr)^2 \]
is again a hamiltonian cycle in $\Cay( \widehat G ; a, b, c)$. Its voltage is
	\[ \voltage(C) 
	= \bigl( (b c)^{2q} \, c a \bigr)^2
	= \bigl( 1 \cdot c a \bigr)^2
	= ( c a )^2
	= (e_3 \cdot e_1 x_p)^2 
	= (e_3 e_1)^2 x_p^2
	= x_p^2
	\neq 1 .\]
So \cref{FGL} applies.

\begin{subcase}
Assume no element of~$S$ inverts~$\cyclic{pq}$.
\end{subcase}
This means that every element of~$S$ has a nontrivial centralizer in~$\cyclic{pq}$. 
However, we also know that $\cyclic{pq} \cap Z(G) = \{1\}$ (see \fullcref{sinG'}{noZ}), which means that no nontrivial subgroup of~$\cyclic{pq}$ is centralized by every element of~$S$. We may also assume that no element of~$S$ centralizes all of~$\cyclic{pq}$ (for otherwise \cref{ElemAbel3-aCent} applies). Therefore, we may assume $a$ and~$b$ centralize~$\cyclic{p}$, and $c$~centralizes~$\cyclic{q}$. Then we may also assume that $a$ and~$b$ do not involve~$x_p$ (otherwise \cref{n1} applies with $s \in \{a,b\}$, $t = s^{-1}$, and $N = \cyclic{p}$). Conjugating by an element of~$\cyclic{p}$, we may assume that $c$ also does not involve~$x_p$. Then no element of~$S$ involves~$x_p$, which contradicts the fact that $S$ generates~$G$.

\begin{case}
Assume  $|S|=4$. 
\end{case}
Let $S_0$ be a $3$-element subset of~$S$ that generates $G/\cyclic{pq}$. By Assumption~\pref{ElemAbel-not8} in the statement of the \lcnamecref{ElemAbel}, we know that $|\langle S_0 \rangle| \neq 8$; therefore $|\langle S_0 \rangle| = 8p$ (perhaps after interchanging $p$ and~$q$). After conjugating, we may assume $\langle S_0 \rangle = (\cyclic2 )^3 \ltimes \cyclic{p}$. 

Let $a$ be the fourth element of~$S$, so $a$ is the only element that involves~$x_q$.
Also choose $b \in S_0$, such that $b$ does not centralize~$\cyclic{q}$. Then $|\langle a, b \rangle|$ is divisible by~$q$.

\begin{subcase}
Assume $a \equiv b \pmod{\cyclic{pq}}$.
\end{subcase}
We may assume $|a b^{-1}| = pq$, for otherwise \cref{n1} applies with $s = a$, $t = b$, and $N = \langle a^{-1} b \rangle \in \{\cyclic{p}, \cyclic{q} \}$. If we write $S = \{a, b, c, d\}$, then the hamiltonian cycle $C_{a, c, d}$ has 4 occurrences of~$a$ or~$a^{-1}$. Therefore, we see from \cref{occur3} that there is a hamiltonian cycle in $\Cay(G;S)$.

\begin{subcase}
Assume $a \not\equiv b \pmod{\cyclic{pq}}$.
\end{subcase}
 Then we may choose an element~$c$ of~$S_0$, such that $\{a,b,c\}$ generates $G/\cyclic{pq}$. So $|\langle a, b, c \rangle| = 8q$. (Recall that $|\langle a,b \rangle|$ is divisible by~$q$.) Since $|\langle S_0 \rangle| = 8p$, we conclude that $|\langle S_0 \rangle  \cap \langle a, b, c \rangle| = 8$. Since $b, c \in S_0$, then $|\langle b, c \rangle|$ is a divisor of~$8$, so, after conjugating by an element of~$\cyclic{p}$, we may assume $\langle b, c \rangle \subseteq P_2$, which means that $b$ and~$c$ do not involve~$x_p$ (and we already know that they do not involve~$x_q$). 
 
 \begin{subsubcase}
 Assume $a$ involves~$x_p$.
 \end{subsubcase}
 Since $|\langle a, b, c \rangle| = 8q$, this implies that $b$ and~$c$ centralize~$\cyclic{p}$. Let $d$ be the other element of~$S_0$. Then $d$ cannot centralize~$\cyclic{p}$, so, after conjugating by an element of~$\cyclic{p}$, we may assume that $d \in P_2$. (This conjugation does not affect $b$ and~$c$, since they centralize~$\cyclic{p}$.) Now $b$, $c$, and~$d$ all belong to~$P_2$, so $|\langle S_0 \rangle| \le 8$. This contradicts the fact that $|\langle S_0 \rangle| = 8p$.
 
 \begin{subsubcase} \label{ElemAbel-S=4-neq-notxp}
Assume $a$ does not involve~$x_p$.
 \end{subsubcase}
 Then $S = \{ e_1 x_q, e_2, e_3, gx_p \}$, where $e_1,e_2,e_3 = (\cyclic{2})^3$, and $g$ is a nontrivial element of~$(\cyclic{2})^3$. In this situation, the GAP computer program in \textsf{8pq-Prop-5-1.gap}  verifies that there is a hamiltonian cycle in $\Cay( \overline G; S )$ whose voltage generates~$\cyclic{pq}$, unless (up to isomorphism) the Cayley graph is described in \cref{ElemException} below. 
  \end{proof}

\begin{lem} \label{ElemException}
Assume
 	\begin{itemize}
	\item $G = (\cyclic2)^3 \ltimes  \cyclic{pq}$, 
	and
	\item $S = \{e_1 x_q, e_2, e_3,  e_1 e_2 x_p \}$, where $\langle e_1,e_2,e_3 \rangle = (\cyclic2)^3$, such that 
		\begin{itemize}
		\item $e_1$ inverts~$\cyclic{p}$ and~$\cyclic{q}$,
		and
		\item $e_2$ and~$e_3$ centralize~$\cyclic{p}$, and invert~$\cyclic{q}$.
		\end{itemize}
	\end{itemize}
Then $\Cay(G; S)$ has a hamiltonian cycle.
\end{lem}

\begin{proof}
Since $e_2$ and~$e_3$ have the same action on~$\cyclic{pq}$, we know that $e_2 e_3^{-1} \in Z(G)$. Therefore \cref{n1} applies with $s = e_2$, $t = e_3$, and $N = \langle e_2 e_3^{-1} \rangle \cong \cyclic{2}$.
\end{proof}

\section{A case where \texorpdfstring{$\overline{G} \cong D_8$}{the quotient is dihedral}} \label{SpecialDihedralSect}

\begin{prop} \label{SpecialDihedral}
Assume:
	\begin{enumerate}
	\item $G = D_8 \ltimes \cyclic{pq}$, 
	\item $|S| = 3$,
	\item $\overline{\phantom{x}} \colon G \to D_8$ is the natural homomorphism with kernel~$\cyclic{pq}$,
	\item $\overline S = \{f, fx_4, f x_4^{-1} \}$, where $f$ is a reflection, and $x_4$ is a rotation of order~$4$,
	\item $f$ centralizes~$\cyclic p$ and inverts~$\cyclic q$,
	\item $x_4$ inverts~$\cyclic p$, and centralizes~$\cyclic q$,
	and
	\item \label{SpecialDihedral-not8}
	there does not exist a subset~$S_0$ of~$S$, such that $|\langle S_0 \rangle | = 8$.
	\end{enumerate}
Then $\Cay(G; S)$ has a hamiltonian cycle.
\end{prop}

\begin{proof}
Write $S = \{s, t, u\}$ with $\overline s = f$, $\overline t = fx_4$, and $\overline u = f x_4^{-1}$. Note that:
	\begin{itemize}
	\item $s$ centralizes~$\cyclic p$ and inverts~$\cyclic q$, 
	whereas
	\item $t$ and~$u$ invert $\cyclic{pq}$.
	\end{itemize}
We have $|t| = |u| = 2$.
Since $|\overline s| = 2$, we may assume that $s$ also has order~$2$ (otherwise \cref{n1} applies with $t = s^{-1}$ and $N = \cyclic{p}$), so we may assume 
	\[ s = f \in D_8 . \]
Since $s$ centralizes~$\cyclic p$, we know that $|\langle s, t \rangle|$ and $|\langle s, u \rangle|$ are not divisible by~$p$. So we must have $|\langle s, t \rangle| = |\langle s, u \rangle| = 8q$ (by condition~\pref{SpecialDihedral-not8} in the statement of the \lcnamecref{SpecialDihedral}). Conjugating by an element of~$\cyclic p$, we may then assume $\langle s, u \rangle = P_2 \ltimes \cyclic q$. Thus, we have
	\[ \text{$t = fx_4 x_p x_q^i$ \ and \ $u = f x_4^{-1} x_q$ \ for some $i \in \ZZ$ with $i \not\equiv 0 \pmod{q}$} . \]
Note that $\langle x_4^2 \rangle = Z(G)$.

Let 
	\[ \widecheck G 
	= \frac{G}{\langle x_4^2 \rangle \times \cyclic{q}} 
	= \langle \widecheck s \rangle \times \langle \widecheck t, \widecheck u \rangle
	\cong \cyclic2  \times D_{2p} , \]
so
	\[ \text{$C_1 = \bigl( (t, u)^{2p} \#, s \bigr)^2$
	\quad and \quad
	$C_2 = \bigl( (u,t)^{2p} \#, s \bigr)^2$}
	\]
are hamiltonian cycles in $\Cay \bigl( \widecheck G ; S \bigr)$ whose voltages are
	\begin{align*}
	\voltage(C_1) 
	&= \bigl( (tu)^{2p} u s \bigr)^2
	\\&= \bigl( (fx_4 x_p x_q^i \cdot f x_4^{-1} x_q)^{2p} \cdot f x_4^{-1} x_q \cdot f \bigr)^2
	\\&= \bigl( (x_4^{2} \, x_p^{-1} \, x_q^{1-i})^{2p} \, x_4 x_q^{-1}  \bigr)^2
	\\&= x_4^2 \, x_q^{2 \bigl( 2p(1 - i) -1 \bigr)}
\intertext{and}
	\voltage(C_2) 
	&= \bigl( (ut)^{2p} t s \bigr)^2
	\\&= \bigl( (f x_4^{-1} x_q \cdot fx_4 x_p x_q^i)^{2p} \cdot fx_4 x_p x_q^i \cdot f \bigr)^2
	\\&= \bigl( (x_4^{2} \, x_p \, x_q^{i-1})^{2p} \, x_4^{-1} x_px_q^{-i}  \bigr)^2
	\\&= x_4^2 \, x_q^{2 \bigl( 2p(i-1) - i \bigr)}
	. \end{align*}
This shows that $\langle \voltage(C) \rangle$ contains $\langle x_4^2 \rangle$.
Hence, we may assume 
	\[ \text{$2p(1 - i) -1 \equiv 0 \pmod{q}$ 
	\quad and \quad
	$2p(i-1) - i \equiv 0 \pmod{q}$,} \]
for otherwise either $\voltage(C_1)$ or $\voltage(C_2)$ generates $\langle x_4^2 \rangle \times \cyclic q$, so \cref{FGL} applies. 
Adding these two congruence's yields $-(1 + i) \equiv 0$, so $i \equiv -1$, which means
	\[ t = fx_4 x_p x_q^{-1} . \]
Also, substituting $i = -1$ into the first congruence tells us that
	\begin{align} \label{D8-0modq} \tag{$\otimes$}
	4p \equiv 1 \pmod{q} 
	. \end{align}

Now, let $\widehat G = G/\cyclic p$. We have
	\[ t u  
	= fx_4 x_p x_q^{-1}  \cdot f x_4^{-1} x_q 
	= x_4^2 x_p^{-1} x_q^2
	\equiv x_4^2 x_q^2 \pmod{\cyclic p}
	, \]
so $|\widehat t \, \widehat u| = 2q$. Since $|t| = |u| = 2$, this implies that $\langle \widehat t, \widehat u \rangle$ is isomorphic to $D_{4q}$, and is a subgroup of index~$2$ in~$\widehat G$. Hence, we have the following hamiltonian cycle in $\Cay( \widehat G; S)$:
	\[ C = \bigl( (t, u)^{2q}\#, s,  (u, t)^{2q} \#, s \bigr) . \]
(It may not be obvious that the walk~$C$ is closed, but that follows from the following calculation of its voltage, which establishes that the terminal vertex of the walk is in~$\cyclic p$.)
Its voltage is
	\begin{align*}
	\voltage(C)
	&=  (t u)^{2q} u s (u t)^{2q} t s
	\\&= (fx_4 x_p x_q^{-1} \cdot  f x_4^{-1} x_q)^{2q} \cdot  f x_4^{-1} x_q \cdot f \cdot(  f x_4^{-1} x_q \cdot fx_4 x_p x_q^{-1})^{2q} \cdot fx_4 x_p x_q^{-1} \cdot f
	\\&= (x_4^2 x_p^{-1} x_q^2)^{2q} \cdot x_4 x_q^{-1} \cdot (x_4^2 x_p x_q^{-2})^{2q} \cdot x_4^{-1} x_p x_q
	\\&= (x_p^{-2q}) \cdot x_4 x_q^{-1} \cdot (x_p^{2q}) \cdot x_4^{-1} x_p x_q
	\\&= x_p^{1-4q}
	. \end{align*}
If this does not generate~$\cyclic{p}$, then $4q \equiv 1 \pmod{p}$.

Much like at the end of \cref{ElemAbelPf-2invert} of the proof of \cref{ElemAbel}, we will show that this leads to a contradiction with~\pref{D8-0modq}. Let $k,\ell \in \ZZ^+$, such that $4p = k q + 1$ and $4q = \ell p + 1$. Then
	\[ q 
	= \frac{\ell p+1}{4} 
	= \frac{\displaystyle \ell \cdot \frac{k q + 1}{4} + 1}{4} 
	= \frac{k \ell q + \ell + 4}{16}
	. \]
This obviously implies $k \ell < 16$ and $\ell + 4 \equiv 0 \pmod{q}$. By symmetry, we also have $k + 4 \equiv 0 \pmod{p}$. Assume, for definiteness, that $p < q$. (The other case is completely analogous.) Since $p \ge 7$ \csee{pq>5}, this implies $q \ge 11$, so $k \ge p - 4 \ge 3$ and $\ell \ge q - 4 \ge 7$. This contradicts the fact that $k \ell < 16$.
	
\end{proof}

\section{Proof of the main theorem} \label{MainProofSect}

This \lcnamecref{MainProofSect} proves the main theorem~\pref{8pqThm}. As described in \cref{ComputerFGL}, most cases are handled by using a computer to do exhaustive case-by-case analysis that finds a hamiltonian cycle in a quotient group of order~$8$. However, some cases were handled in previous sections (especially \cref{SpecialDihedralSect,ElemAbelSect}), and a small amount of additional work is done by hand in this \lcnamecref{MainProofSect}.

The assumptions and notation of \cref{AssumpSect} are in effect.

\begin{notation} \label{d(G)defn}
Let $d(\overline G)$ be the cardinality of an irredundant generating set of~$\overline G$. Since $\overline G$ has prime-power order, it is well known that this is well-defined, independent of the choice of the irredundant generating set (by the Burnside Basis Theorem \cite[Thm.~12.2.1, p.~176]{Hall-ThyOfGrps}). Specifically:
	\begin{enumerate}
	\item $d(\cyclic8 ) = 1$,
	\item $d(\cyclic4 \times \cyclic2 ) = d(D_8) = d(Q_8) = 2$,
	and
	\item $d ( \cyclic2 \times \cyclic2 \times \cyclic2 ) = 3$
	\end{enumerate}
(where $Q_8 = \{\pm1, \pm i, \pm j, \pm k\}$ is the quaternion group of order~$8$).
\end{notation}

The following observation is elementary (and well known):

\begin{lem} \label{Svsd(Gbar)}
We have
	\[ d(\overline G) \le |S| \le d(\overline G) + 2 . \]
\end{lem}

\begin{proof}
Since $S$ generates~$G$, it must generate $G/\cyclic{pq} = \overline G$. Therefore, it contains a subset~$S_0$ that is an irredundant generating set of~$\overline G$. Since $|S_0| = d(\overline G)$, this establishes that $d(\overline G) \le |S|$.

To establish the other inequality, let $s_1, s_2, \ldots, s_k$ be a list of the elements of~$S$ that are not in~$S_0$, so $|S| = d(\overline G)+ k$. Since $S$ is irredundant, we have
	\[ \langle S_0 \rangle \subsetneq \langle S_0, s_1 \rangle\subsetneq \langle S_0, s_1, s_2 \rangle \subsetneq \cdots \subsetneq  \langle S_0, s_1, s_2, \ldots, s_k \rangle . \] 
Thus, if we let 
	\[ \text{$m_i = |\langle S_0, s_1, s_2, \ldots, s_i \rangle : \langle S_0, s_1, s_2, \ldots, s_{i-1} \rangle|$ \ for $i = 1,2, \ldots, k$}, \]
then $m_i > 1$, and
	\[ m_1 m_2 \cdots m_k = |G : \langle S_0 \rangle| . \]
However, since $S_0$ generates $\overline G$, we know that $|\langle S_0 \rangle|$ is a multiple of $|\overline G| = 8$. Therefore, $|G : \langle S_0 \rangle|$ is a divisor of $|G|/8 = pq$, so it cannot be written as a product of more than two nontrivial factors. We conclude that $k \le 2$.
\end{proof}

We first handle the smallest value of~$|S|$ that is consistent with \cref{Svsd(Gbar)}:

\begin{prop} \label{MinimalInQuotient}
If $|S| = d(\overline G)$, then $\Cay(G; S)$ has a hamiltonian cycle.
\end{prop}

\begin{proof}
We know that $|S| \neq 1$ (because \cref{GIsSemiprod} implies that $G$ is nonabelian), so $d(\overline G) \neq 1$. Also, \fullcref{ElemAbel}{S=3} applies if $|S| = 3$. Therefore, we may assume $d(\overline G) = 2$. We may also assume $\overline G$ is abelian, for otherwise $G/G' \cong \cyclic2  \times \cyclic2 $, so \cref{abelianization=Z2xZ2} applies. Since the only abelian groups of order~$8$ are $\cyclic8$, $\cyclic4 \times \cyclic2$, and~$(\cyclic2)^3$, we conclude that 
	\[ \text{$\overline G \cong \cyclic4 \times \cyclic2$ (and $|S| = 2$).} \]

Let $s \in S$, such that $|\overline{s}| = 4$, and let $t$ be the other element of~$S$. Then $(s^{-3}, t^{-1}, s^3, t)$ is a hamiltonian cycle in $\Cay(\overline G; S)$. Its voltage is $[s^3, t]$, and, since $\gcd \bigl( 3, |G| \bigr) = 1$, we know from \csee{genG'} that $\langle [s^3, t] \rangle = G'$. Therefore \cref{FGL} applies.
\end{proof}

We now handle the largest possible value of~$|S|$:

\begin{prop} \label{2extra}
If $|S| = d(\overline G) + 2$, then $\Cay(G; S)$ has a hamiltonian cycle.
\end{prop}

\begin{proof}
Let $S_0$ be an irredundant generating set of~$\overline G$ that is contained in~$S$, so $|S_0| = d(\overline G)$ and $|\langle S_0 \rangle|$ is divisible by~$8$. Then we may assume that the Sylow $2$-subgroup~$P_2$ is contained in $|\langle S_0 \rangle|$ (after replacing it by a conjugate).

We claim that $\langle S_0 \rangle = P_2$, and that we may write $S = S_0 \cup \{a x_p, b x_q\}$, where $a$ and~$b$ are elements of~$P_2$. To see this, we argue much as in the proof of \cref{Svsd(Gbar)}. Let $s$ and~$t$ be the two elements of~$S$ that are not in~$S_0$. Since $S$ is irredundant, we know that
	\[ \text{$\langle S_0 \rangle \subsetneq \langle S_0, s \rangle \subsetneq G$
	\quad and \quad 
	$\langle S_0 \rangle \subsetneq \langle S_0, t \rangle \subsetneq G$.} \]
On the other hand, any subgroup whose order is divisible by~$p$ must contain~$\cyclic{p}$ (because, being normal, this is the only Sylow $p$-subgroup of~$G$), and, similarly, any subgroup whose order is divisible by~$q$ must contain~$\cyclic{q}$, so it is easy to see that the only proper subgroups of~$G$ that contain~$P_2$ are 
	\[ \text{$P_2$, $P_2 \ltimes \cyclic{p}$, and $P_2 \ltimes \cyclic{q}$} . \]
We conclude (perhaps after interchanging $p$ and~$q$) that we have
	\[ \text{$\langle S_0 \rangle = P_2$,
	$\langle S_0, s \rangle = P_2 \ltimes \cyclic{p}$,
	and
	$\langle S_0, t \rangle = P_2 \ltimes \cyclic{q}$.} \]
This completes the proof of the claim.

Running the GAP computer program in \textsf{8pq-Prop-7-4.gap} verifies in all cases that there is a hamiltonian cycle in $\Cay(\overline G ; S)$ whose voltage generates~$\cyclic{pq}$, so \cref{FGL} applies.
\end{proof}

The preceding three 
results allow us to make the following assumption:

\begin{assump} \label{S=d+1}
Assume $|S| = d(\overline G) + 1$.
\end{assump}

Let us now consider three additional special cases.

\begin{lem} \label{S0=8Special}
Assume $G = (\cyclic2)^3 \ltimes \cyclic{pq}$, and $S = \{a,b,c,  a b x_p x_q\}$, where $\langle a, b, c \rangle = (\cyclic2)^3$, such that 
		\begin{itemize}
		\item $a$ inverts~$\cyclic p$ and centralizes~$\cyclic q$,
		and
		\item $b$ and~$c$ centralize~$\cyclic p$ and invert~$\cyclic q$.
		\end{itemize}
Then $\Cay(G; S)$ has a hamiltonian cycle.
\end{lem}

\begin{proof}
Let $N = \langle b^{-1} c \rangle$.
Since $b$ and~$c$ have the same action on~$G'$, we know that $b^{-1} c \in Z(G)$, so $N$ is a normal subgroup. Then, since $N$ has order~$2$ (because it is a nontrivial, cyclic subgroup of $(\cyclic2)^3$), we see from \cref{n1} (with $s = a$ and $t = b$) that $\Cay(G; S)$ is hamiltonian.
\end{proof}

\begin{prop} \label{S0=8}
If $S$ contains a subset~$S_0$, such that $|\langle S_0 \rangle| = 8$, then $\Cay(G; S)$ has a hamiltonian cycle.
\end{prop}

\begin{proof}
Let $P_2 = \langle S_0 \rangle$, so $P_2$ is a Sylow $2$-subgroup of~$G$. 
By \cref{S=d+1}, we have $S = S_0 \cup \{g x_p x_q\}$, for some $g \in P_2$. Also note that we may assume $g \notin P_2'$, for otherwise \fullcref{sinG'}{ham} applies. 
In this situation, the computer program in \textsf{8pq-Prop-7-7.gap} establishes that either \cref{S0=8Special} applies, or there is a hamiltonian cycle in $\Cay(G/\cyclic{pq} ; S)$ whose voltage generates~$\cyclic{pq}$ (so \cref{FGL} applies). 
In either case, there is a hamiltonian cycle in $\Cay(G; S)$.
\end{proof}

\begin{cor} \label{Z8}
If $d(\overline G) = 1$, then every connected Cayley graph on~$G$ has a hamiltonian cycle.
\end{cor}

\begin{proof}
Let $S$ be an irredundant generating set of~$G$, and let $S_0 \subseteq S$, such that $S_0$ is an irredundant generating set of~$\overline G$. Then 
	\[ |S_0| = d(\overline G) = 1 , \]
so we may write $S_0 = \{s\}$. Then (since $G \cap \cyclic{pq}$ is trivial) it is easy to see that $|s| = 8$ (cf.~\cite[Lem.~2.16]{Maghsoudi-6pq}), so \cref{S0=8} applies.
\end{proof}

Note that:
	\begin{itemize}
	\item if $d(\overline G) = 1$, then \cref{Z8} applies,
	and
	\item if $d(\overline G) = 3$, then either \fullcref{ElemAbel}{not8} or \cref{S0=8} applies.
	\end{itemize}
Therefore, the following result completes the proof of the main theorem~\pref{8pqThm}.

\begin{prop} \label{S=3}
If $d(\overline G) = 2$, then $\Cay(G; S)$ has a hamiltonian cycle.
\end{prop}

\begin{proof}
Since $d(\overline G) = 2$, we may let $\{a,b\}$ be an irredundant generating set of $\overline G$ that is contained in~$S$. By \cref{{S=d+1}}, we know that $S \neq \{a,b\}$, so the fact that $S$ is irreducible implies $\langle a,b \rangle \neq G$. Therefore, we may assume $|\langle a,b \rangle| = 8q$ (perhaps after interchanging $p$ and~$q$), for otherwise \cref{S0=8} applies. So $\langle a,b \rangle = P_2 \ltimes \cyclic{q}$ (after passing to a conjugate). Since $\cyclic{q} \nsubseteq Z(G)$ (see \fullcref{sinG'}{noZ}), we may assume $a$ does not centralize~$\cyclic{q}$ (perhaps after interchanging $a$ and~$b$). Then, after conjugating by an element of~$\cyclic{q}$, we may assume $a \in P_2$. Then $b = \overline{b} x_q$, for some $\overline{b} \in P_2$.

Let $c$ be the third element of~$S$. We may write $c = \overline{c} x_p x_q^i$, where $\overline{c} \in P_2$ and $i \in \ZZ$, and we may assume $c \notin G'$, for otherwise \fullcref{sinG'}{ham} applies.

In this situation, the GAP computer program in \textsf{8pq-Prop-7-9.gap} establishes that either \cref{SpecialDihedral} (or \cref{S0=8}) applies, or there is a hamiltonian cycle in $\Cay(G/\cyclic{pq} ; S)$ whose voltage generates~$\cyclic{pq}$ (so \cref{FGL} applies).
(See \cref{ComputerFGL-TwoInvolved} of \cref{ComputerFGL} for an explanation of the basic logic of the program.)
In either case, there is a hamiltonian cycle in $\Cay(G; S)$.
\end{proof}

\end{document}